\documentclass[a4paper]{amsart}
\usepackage[utf8]{inputenc}

\usepackage{newunicodechar}
\newunicodechar{ō}{\=o}

\usepackage{amsmath}
\usepackage{amssymb}
\usepackage{amsthm}

\usepackage{enumitem}

\usepackage{cite}
\usepackage{hyperref}
\usepackage{cleveref}

\newtheorem{thm}{Theorem}
\newtheorem{lemma}[thm]{Lemma}

\newtheorem{definition}[thm]{Definition}

\newtheorem{remark}[thm]{Remark}
\newtheorem{proposition}[thm]{Proposition}

\crefname{thm}{theorem}{theorems}
\Crefname{thm}{Theorem}{Theorems}

\newcommand{\norm}[1]{\ensuremath{\|#1\|}}
\newcommand{\dd}{\mathrm{d}}
\newcommand{\T}{\ensuremath{\mathbb{T}}} 
\newcommand{\R}{\ensuremath{\mathbb{R}}} 
\newcommand{\N}{\ensuremath{\mathbb{N}}} 
\newcommand{\E}{\ensuremath{\mathbb{E}}} 

\begin{document}

\title[Convergence for the kinetic Fokker-Planck equation]
{Convergence to equilibrium for the kinetic Fokker-Planck
  equation on the torus}
\author{Helge Dietert}
\address{Department of Pure Mathematics and Mathematical Statistics\\
University of Cambridge\\
Wilberforce Road\\
Cambridge CB3 0WA, UK}
\email{H.G.W.Dietert@maths.cam.ac.uk}
\thanks{The authors were supported by the UK Engineering and Physical
Sciences Research Council (EPSRC) grant EP/H023348/1 for the
University of Cambridge Centre for Doctoral Training, the Cambridge
Centre for Analysis.}
\author{Josephine Evans}
\address{Department of Pure Mathematics and Mathematical Statistics\\
University of Cambridge\\
Wilberforce Road\\
Cambridge CB3 0WA, UK}
\email{jahe2@cam.ac.uk}
\author{Thomas Holding}
\address{Department of Pure Mathematics and Mathematical Statistics\\
University of Cambridge\\
Wilberforce Road\\
Cambridge CB3 0WA, UK}
\email{T.J.Holding@maths.cam.ac.uk}
\keywords{convergence to equilibrium;
  hypocoercivity;
  Monge-Kantorovich-Wasserstein $\mathcal{W}_2$ distance;
  Fokker-Planck equation; torus;
  co-adapted couplings}
\subjclass[2010]{60J60;
  35Q84, 60H30, 82C31}
\begin{abstract}
  We study convergence to equilibrium for the kinetic Fokker-Planck
  equation on the torus. Solving the stochastic differential equation,
  we show exponential convergence in the Monge-Kantorovich-Wasserstein
  $\mathcal{W}_2$ distance. Finally, we investigate if such a
  coupling can be obtained by a co-adapted coupling, and show that
  then the bound must depend on the square root of the initial distance.
\end{abstract}
\maketitle

\section{Introduction}

The kinetic Fokker-Planck equation, also known as the Kramers equation,
is a basic model for the spreading of a solute due to interaction with
the fluid background. It is derived from Langevin dynamics, where the
time scale of observation is much larger than the correlation time of
the solute-fluid interactions (see e.g. \cite{zwanzig2001nonequilibrium}). In the context of fixed
random scatters the similar linear Landau equation can rigorously
be derived in the weak coupling limit, see
\cite{Kirkpatrick2009rigorous} and references within. We focus on the
case that the space variable is in the torus $\T = \mathbb{R}/(2\pi L\mathbb{Z})$ of
length $2\pi L$. The kinetic Fokker-Planck equation describes the law of a
particle moving in the phase space $\T \times \R$ whose location in
the phase space is $(X_t,V_t)$ and evolves as
\begin{equation}
  \label{eq:kinetic-fokker-planck-SDE}
  \left\{
    \begin{aligned}
      &\dd X_t = V_t \dd t, \\
      &\dd V_t = -\lambda V_t \dd t + \dd W_t,
    \end{aligned}
  \right.
\end{equation}
where $\dd W_t$ is a standard white noise. The corresponding measure
$\mu_t$ on $\T \times \R$ evolves as
\begin{equation}
  \label{eq:kfp}
  \partial_t \mu_t + v \partial_x \mu_t = \partial_v [ \lambda v
  \mu_t + \frac{1}{2} \partial_v \mu_t ],
\end{equation}
where this equation is considered in the weak sense.

We expect that the measure $\mu_t$ spreads out over time and
eventually reaches the uniform measure which is the unique stationary
state. The problem of convergence to equilibrium has been studied in
different metrics before (see
e.g. \cite{Villani-hypocoercivity-2006,Mouhot-Mischler-Exponential-stability-2014})
and forms a key example of hypocoercivity.

To the best of the authors' knowledge convergence in the
Monge-Kantorovich-Wasserstein (MKW) distance $\mathcal{W}_2$ has not
been solved and is the object of this paper.

The MKW distance comes from optimal transport and is defined as
\begin{equation*}
  \mathcal{W}_2(\mu, \nu) = \inf_{\pi \in \Pi_{\mu, \nu}}\left( \int
  |x-y|^2 \dd\pi(x,y) \right)^{1/2},
\end{equation*}
where $\Pi_{\mu,\nu}$ is the set of all couplings between $\mu$ and
$\nu$. This metric is very useful as it allows to understand the
Fokker-Planck equation as gradient flow \cite{jordon1998variational},
see \cite{villani2009optimal} for a general review.

In the spatially homogeneous case, i.e. only considering $V_t$, this
is an Orstein-Uhlenbeck process for which exponential convergence to
equilibrium has been proven in the Wasserstein distance
\cite{bolley2012convergence}. In a stochastic framework the
convergence can be proved by coupling the noise and using the fact
that the dependence on the initial data decays over time, which in an
analytic setting translates to a functional inequality for the time
derivative showing that the evolution is a contraction semigroup.

In the case where there is also a spatial variable, the same coupling
approach works if the spatial variable evolves in a confining
potential. However, in our case on the torus the spatial distance will
not decay if we just couple the velocities.

Solving the stochastic evolution, we are still able to show
exponential decay of the distance between two solutions.
\begin{thm}
  \label{thm:non-markovian-coupling}
  If $\mu_t$ and $\nu_t$ are two solutions to the kinetic
  Fokker-Planck equation \eqref{eq:kfp}, then we have
  \begin{equation*}
    \mathcal{W}_2(\mu_t, \nu_t) \leq \left( e^{-\lambda t} + c\,
      e^{-t/4\lambda^2 L^2} \right) \mathcal{W}_2(\mu_0, \nu_0)
  \end{equation*}
   for a constant $c$ only depending on $L$.
\end{thm}

The key idea is that, after fixing the net effect of the velocity
noise, the spatial variable has enough randomness left to allow such a
coupling. This approach is not based on a functional inequality which
is integrated over time and in fact the evolution is not a contraction
semigroup.
\begin{thm}
  \label{thm:no-contraction}
  There exists no $\gamma > 0$ such that for all solutions $\mu_t$ and
  $\nu_t$ to the kinetic Fokker-Planck equation \eqref{eq:kfp} we have
  \begin{equation*}
    \mathcal{W}_2(\mu_t, \nu_t)
    \leq e^{-\gamma t}\, \mathcal{W}_2(\mu_0, \nu_0)
  \end{equation*}
  for all $t \ge 0$.
\end{thm}

This shows that the generator is not coercive but only hypocoercive in
$\mathcal{W}_2$.

In probability theory a classical approach to such convergence results
is the construction of a coupling \cite{Lindvall1992}. For this, random
variables $(X^i_t,V^i_t)$ are constructed for $t \in \R^+$ and $i=1,2$
such that $(X^1_t,V^1_t)$ has law $\mu_t$ and $(X^2_t,V^2_t)$ has law
$\nu_t$. Then for $t \in \R^+$ the coupling
$((X^1_t,V^1_t),(X^2_t,V^2_t))$ gives an upper bound of the MKW
distance $\mathcal{W}_2(\mu_t,\nu_t)$.

The stochastic differential equation
\eqref{eq:kinetic-fokker-planck-SDE} motivates to look at couplings
where $(X^i_t,V^i_t)$ are continuous Markov processes with initial
distribution $\mu_0$ and $\nu_0$, respectively, and whose transition
semigroup is determined by \eqref{eq:kinetic-fokker-planck-SDE}. For
such couplings we can consider a more restrictive class of couplings.
\begin{definition}[co-adapted coupling]
  The coupling $((X^1_t,V^1_t),(X^2_t,V^2_t))$ is co-adapted if, for
  $i = 1,2$, under the filtration $\mathcal{F}$ generated by the
  coupling $((X^1_t,V^1_t),(X^2_t,V^2_t))$, the process $(X^i_t,V^i_t)$
  is a continuous Markov process whose transition semigroup is
  determined by \eqref{eq:kinetic-fokker-planck-SDE}.
\end{definition}

This is an important subclass of couplings, which contains many
natural couplings, and an even more restrictive subclass is the class of
Markovian couplings, where additionally the coupling itself is imposed
to be Markovian. The existence and obtainable convergence behaviour
under this restriction has already been studied in different cases,
e.g. \cite{kuwada2009characterization,krzysztof2000efficient,chen1994optimal}. Note
that the co-adapted coupling is equivalent to the condition that the
filtration generated by $(X^i_t,V^i_t)$ is immersed in the filtration
generated by the coupling, which motivates Kendall \cite{Kendall2015}
to call such couplings \emph{immersed couplings}.

By adapting the reflection/synchronisation coupling, we can still
obtain exponential convergence but with a loss  in dependence on the initial data.
\begin{thm}\label{thm:markovian-coupling-existence}
  Given initial distributions $\mu_0$ and $\nu_0$, there exists a
  co-adapted coupling $((X_t^1,V_t^1),(X_t^2,V_t^2))$ such that
 \begin{align*}
   \mathcal{W}_2(\mu_t, \nu_t)
   &\le
     \left(\E\left[|X^1_t-X^2_t|_{\T}^2+(V^1_t-V^2_t)^2\right]\right)^{1/2} \\
   &\le C \beta (t) (\sqrt{\mathcal{W}_2(\mu_0, \nu_0)}+\mathcal{W}_2(\mu_0,\nu_0)),
  \end{align*}
  where
  \begin{equation*}
    \beta(t)=\begin{cases}
      e^{-\min(2\lambda,1/(2\lambda^2 L^2))t}& 4L^2\lambda^3\ne1\\
      e^{-2\lambda t}(1+t)& 4L^2\lambda^3=1
    \end{cases}
  \end{equation*}
  and $C$ is a constant that depends only on $\lambda$ and $L$.
\end{thm}
Here we used the notation $|X^1_t-X^2_t|_{\T}$ to emphasis that this
is the distance on the torus $\T$. In fact the filtration generated by
$(X^1,V^1)$ and $(X^2,V^2)$ agree which Kendall \cite{Kendall2015}
calls an equi-filtration coupling.
\begin{remark}
  This achieves the same exponential decay rate as the non-Markovian
  argument, except for the case $4L^2\lambda^3=1$, when the spatial
  and velocity decay rates coincide and we have an addition polynomial
  factor.
\end{remark}

In general the loss in the dependence is necessary.
\begin{thm}\label{thm:markovian-coupling-optimality}
  Suppose there exists a function $\alpha : \R^+ \mapsto \R^+$
  and a constant $\gamma > 0$ such that for all initial distributions
  $\mu_0$ and $\nu_0$ there exists a co-adapted coupling
  $((X_t^1,V_t^1),(X_t^2,V_t^2))$ such that
  \begin{equation*}
    \left(
      \mathbb{E}\left[ |X_t^1-X_t^2|^2_{\mathbb{T}} + (V_t^1-V_t^2)^2
      \right]
    \right)^{1/2}
    \le \alpha(\mathcal{W}_2(\mu_0,\nu_0)) e^{-\gamma t}.
  \end{equation*}
  Then there exists a constant $C$ such that for $z \in (0,\pi L]$
  we have the following lower bound on the dependence on the initial
  distance
  \begin{equation*}
    \alpha(z) \ge C \sqrt{z}.
  \end{equation*}
\end{thm}

The idea is to focus on a drift-corrected position on the torus, which
evolves as a Brownian motion. By stopping the Brownian motion at a
large distance we can then prove the claimed lower bound.

This shows that a simple hypocoercivity argument on a Markovian
coupling cannot work. Precisely, there cannot exist a semigroup $P$ on
the probability measures over $(\T \times \R)^{\times 2}$, whose
marginals behave like the solution of
\eqref{eq:kinetic-fokker-planck-SDE} and which satisfies
$H(P_t(\pi)) \le c H(\pi) e^{-\gamma t}$ for
$H^2(\pi) = \int [(X^1-X^2)^2 + (V^1-V^2)^2] \dd
\pi(X^1,V^1,X^2,V^2)$.
Otherwise, the Markov process associated to $P$ would be a coupling
contradicting \Cref{thm:markovian-coupling-optimality}.

\subsection{Acknowledgements}

The authors would like to thank Clément Mouhot for the initial
discussion to look into the problem.

\section{Set up}
The stochastic differential equation
\eqref{eq:kinetic-fokker-planck-SDE} has the explicit solution
\begin{equation}\label{eq:setup-explicit-solution}
\begin{aligned}
X_t &= X_0 + \frac{1}{\lambda}(1-e^{-\lambda t})V_0 + \int_0^t
      \frac{1}{\lambda}(1-e^{-\lambda(t-s)}) \dd W_s,\\
V_t&= e^{-\lambda t} V_0 + \int_0^t e^{-\lambda(t-s)} \dd W_s,
\end{aligned}
\end{equation}
where $W_t$ is the common Brownian motion. In this we separate the
stochastic driving as $(A_t,B_t)$ given by the stochastic integrals
\begin{equation*}
\begin{aligned}
A_t &= \int_0^t \frac{1}{\lambda} (1-e^{-\lambda(t-s)}) \dd W_s,\\
B_t &= \int_0^t e^{-\lambda (t-s)} \dd W_s,
\end{aligned}
\end{equation*}
which evolve over $\R$ with the common Brownian motion $W_t$. By Itō's
isometry $(A_t,B_t)$ is a Gaussian random variable with covariance
matrix $\Sigma(t)$ given by
\begin{align}
\label{eq:covariance-aa}
\Sigma_{AA}(t) &= \frac{1}{\lambda^2} \left[ t - \frac{2}{\lambda} ( 1- e^{-\lambda t}) + \frac{1}{2\lambda} ( 1- e^{-2\lambda t}) \right],\\
\label{eq:covariance-ab}
\Sigma_{AB}(t) &= \frac{1}{\lambda^2} \left[ (1-e^{-\lambda t} ) - \frac{1}{2} ( 1- e^{-2\lambda t}) \right],\\
\label{eq:covariance-bb}
\Sigma_{BB}(t) &= \frac{1}{2\lambda}(1-e^{-2\lambda t}).
\end{align}
From this we calculate that the conditional distribution of $A_t$
given $B_t$ is a Gaussian with variance
$\Sigma_{AA}(t)-2\Sigma^2_{AB}(t)\Sigma^{-1}_{BB}(t)$ and mean given
by
\begin{equation*}
  \mu_{A|B}(t,b) = \Sigma_{AB}(t)\Sigma^{-1}_{BB}(t)b.
\end{equation*}
We write $g_{A|B}$ for the conditional density of $A$ given $B$ and
$g_B$ for the marginal density of $B$. Hence
\begin{equation}
  \label{eq:g-definition}
  g(t,a,b)=g_{A|B}(t,a,b)g_{B}(t,b)
\end{equation}
is the joint density of $A$ and $B$.

The last part of the set up is the change of variables we will need
for the Markovian coupling. We define new coordinates $(Y,V)$ by
taking the drift away
\begin{equation}
  \label{eq:drift-correct-coordinates}
  \left\{
    \begin{aligned}
      Y&=X+\frac1\lambda V,\\
      V&=V.
    \end{aligned}
  \right.
\end{equation}
The motivation for this change is the explicit formulas found in
\eqref{eq:setup-explicit-solution} from which we see that $Y$ is the
limit as $t \rightarrow \infty$ of $X_t$ without additional noise.  In
the new variables, \eqref{eq:kinetic-fokker-planck-SDE} becomes
\begin{equation*}
  \left\{
    \begin{aligned}
      \dd Y_t &= \frac1\lambda \dd W_t, \\
      \dd V_t &= -\lambda V_t \dd t + \dd W_t,
    \end{aligned}
  \right.
\end{equation*}
for the common Brownian motion $W_t$. Note that the motion of $Y_t$
does not depend explicitly upon $V_t$ and is a Brownian motion on the
torus.

It remains to show that these new coordinates define an equivalent
norm on $\mathbb{T}\times\mathbb{R}$. This follows from the triangle
inequality and we have
\begin{equation*}
  |X^1-X^2|_\T+|V^1-V^2|\le |Y^1-Y^2|_\T+\left(1+\frac1\lambda\right)|V^1-V^2|
\end{equation*}
and the other direction is similar. Thus, the two norms are equivalent
up to a constant factor that depends only on $\lambda$.

\section{Non-Markovian Coupling}
We wish to estimate how much the spatial variable will spread out over
time. We will then use this to construct a coupling at a fixed time
$t$ which exploits the fact that a proportion of the spatial density
is distributed uniformly. In order to do this we give a lemma on the
spreading of a Gaussian density wrapped on the torus.
\begin{lemma}
  \label{thm:spreading-estimate}
For $\sigma > 0$ consider the Gaussian density $h$ on $\mathbb{R}$ given by
\[ h(x) = \frac{1}{\sqrt{2\pi \sigma^2}} e^{-x^2/2\sigma^2} \] and wrap it onto the torus $\mathbb{T}$, i.e. define the density $Qh$ on $\mathbb{T}$ by
\begin{equation}
\label{eq:Q-definition}
(Qh)(x) = \sum_{n \in \mathbb{Z}} h(x+2\pi L n).
\end{equation}
We have the following estimate on the spatial spreading
\[ Qh(x) \geq \frac{\beta}{2\pi L}  \] where
\[ 1-\beta = \frac{2 e^{-\sigma^2/2L^2}}{1-e^{-\sigma^2/2L^2}}. \]
\end{lemma}
\begin{proof}
  By the definition of $Q$, the Fourier transform of $Qh$ is for
  $k \in \N$ given by
  \begin{align*}
    (\mathcal{F}Qh)(k) &= \int_{\mathbb{T}} \sum_{n \in \mathbb{Z}}
                         h(x+2\pi L n) e^{ikx/L}\dd x \\
                       &= \int_{\mathbb{R}} h(x) e^{ikx/L} \dd x\\
                       &=\exp \left( -\frac{k^2 \sigma^2}{2L^2} \right)
  \end{align*}
  where we have used the well-known Fourier transformation of a
  Gaussian.

By the Fourier series we find that, for any $x \in \mathbb{T}$, we have
\[ Qh(x)-\frac{\beta}{2\pi L} = \frac{1}{2\pi L} \sum_{|k| \geq 1}
e^{-k^2 \sigma^2/2L^2 - ikx/L} + \frac{1-\beta}{2\pi L}. \]
We want this to be positive. Therefore it is sufficient to show that
\[ \left| \sum_{|k| \geq 1} e^{-k^2 \sigma^2 /2L^2 -ikx/L}\right| \leq 1-\beta. \]
We estimate the left hand side by
\[ \left| \sum_{|k| \geq 1} e^{-k^2 \sigma^2 /2L^2 -ikx/L}\right| \leq 2\sum_{k \geq 1} e^{-k\sigma^2/2L^2} =1-\beta\]
where the final equality follows from summing the geometric series.
\end{proof}
We can now use this to construct a coupling at time $t$. We will use this coupling to prove exponential decrease in the Wasserstein distance.

\begin{lemma}
  \label{thm:coupling-construction}
  Let $\mu_0, \nu_0$ be probability distributions on
  $\mathbb{T} \times \mathbb{R}$ and let
  $((X_0^1, V_0^1), (X_0^2, V_0^2))$ be a coupling between them. Let
  $t \geq 0$ and $\beta>0$ be such that for all $b\in\mathbb{R}$,
  \begin{equation*}
    (Qg_{A|B}(t,\cdot,b))(a) \geq \frac{\beta}{2\pi L},
  \end{equation*}
  where $g_{A|B}$ and $Q$ are defined by \eqref{eq:g-definition} and \eqref{eq:Q-definition} respectively. Furthermore, let $\mu_t$ respectively $\nu_t$ be the distribution of
  the solution to the Fokker-Plank equation \eqref{eq:kfp} with
  initial data $\mu_0$ and $\nu_0$ respectively after time t. Then
  there exists a coupling $((X_t^1, V_t^1), (X_t^2, V_t^2))$ between
  $\mu_t$ and $\nu_t$ satisfying
  \[ \mathbb{E}\left[(V_t^1-V_t^2)^2 \right] = e^{-2\lambda t} \mathbb{E} \left[ (V_0^1-V_0^2)^2  \right] \]
  and
  \begin{equation*}
    \mathbb{E}\left[ |X_t^1 - X_t^2|_\T^2 \right] \leq  2(1-\beta) \mathbb{E}\left[ |X_0^1 - X_0^2|_\T^2 +\frac{1}{\lambda^2}(V_0^1-V_0^2)^2 \right].
  \end{equation*}
\end{lemma}
\begin{proof}
  Let us construct such a coupling. Split the distribution $Qg_{A|B}$
  as \[Qg_{A|B}(t,a,b) = \frac{\beta}{2\pi L} + (1- \beta) s(t,a,b).\]
  Then by assumption $s$ is again a probability density for the
  variable $a$ on the torus $\mathbb{T}$. Let $B$ be an independent
  random variable with density $g_{B}(t,b)$, let $Z$ be an independent
  uniform random variable over $[0,1]$ and let $U$ be an independent
  uniform random variable over the torus. Finally let $S$ be a random
  variable with density $s(t,\cdot, B)$ only depending on $B$.

With this define the random parts $A^1,A^2$ of $X^1_t,X^2_t$ as
\begin{align*}
  A^1=& 1_{Z \leq \beta} \left[ U - X^1_0 -
        \frac{1}{\lambda}(1-e^{-\lambda t})V^1_0 \right]
        + 1_{Z > \beta} S, \\
  A^2 = & 1_{Z \leq \beta} \left[ U-X^2_0 -
          \frac{1}{\lambda}(1-e^{-\lambda t})V^2_0 \right]
          + 1_{Z > \beta} S.
\end{align*}
By construction $(A^1,B)$ and $(A^2, B)$ both have law with density $g(t,a,b)$ so that $(X^1_t,V^1_t)$ defined by
\begin{align*}
X^1_t &= X_0^1 + \frac{1}{\lambda}(1-e^{-\lambda t})V_0^1 + A^1, \\
V_t^1 &= e^{-\lambda t} V_0^1 +B,
\end{align*}
has law $\mu_t$, and $(X^2_t,V^2_t)$ defined by
\begin{align*}
X_t^2 &= X_0^2 + \frac{1}{\lambda}(1-e^{-\lambda t}) V_0^2 + A^2,\\
V_t^2 &= e^{-\lambda t}V_0^2 + B,
\end{align*} has law $\nu_t$.

Hence this is a valid coupling and we find
\[ \mathbb{E} \left[ (V_t^1-V_t^2)^2 \right] = e^{-2\lambda t} \mathbb{E} \left[ (V_0^1-V_0^2)^2 \right] \]
and
\[ \mathbb{E} \left[|X_t^1-X_t^2|_\T^2 \right] = (1- \beta) \mathbb{E} \left[ \left| X_0^1 - X_0^2 + \frac{1}{\lambda}(1-e^{-\lambda t})(V_0^1-V_0^2)\right|_\T^2 \right] \] and we can use Young's inequality to find the claimed control.
\end{proof}
We now put these two lemmas together to prove
\Cref{thm:non-markovian-coupling}, which states exponential
convergence in the MKW $\mathcal{W}_2$ distance.
\begin{proof}[Proof of \Cref{thm:non-markovian-coupling}]
  Given any initial coupling of $((X_0^1,V_0^1), (X_0^2, V_0^2))$, we
  can use \Cref{thm:coupling-construction} to obtain a coupling
  $((X_t^1,V_t^1), (X_t^2, V_t^2))$ of $\mu_t$ and $\nu_t$. From
  explicitly calculating the variance of the distribution of $A|B$
  using \eqref{eq:covariance-aa}, \eqref{eq:covariance-ab},
  \eqref{eq:covariance-bb}, we see that the variance grows
  asymptotically as $t/\lambda^2$. Hence by
  \Cref{thm:spreading-estimate} we can choose $\beta$ so that $1-\beta
  \rightarrow 0$ exponentially fast with rate $1/{2\lambda^2
    L^2}$. This, combined with the control from the second lemma,
  shows that
  \[ \mathbb{E} \left[ (V_t^1-V_t^2)^2 + |X_t^1-X_t^2|_\T^2  \right]
  \leq \left( e^{-2\lambda t} + ce^{-t/2\lambda^2 L^2} \right)
  \mathbb{E} \left[ (V_0^1-V_0^2)^2 + |X_0^1-X_0^2|_\T^2  \right]. \]
  Taking the infimum over all possible couplings at time 0 gives the desired result.
\end{proof}

The explicit solution also allows to prove that the evolution is not a
contraction semigroup.
\begin{proof}[Proof of \Cref{thm:no-contraction}]
  We will prove the theorem by contradiction. Suppose $\gamma > 0$ and
  let $a \not = b$ be two distinct points on the torus. Consider the
  initial measures
  \begin{equation*}
    \mu_0 = \delta_{x=a} \delta_{v=0}
  \end{equation*}
  and
  \begin{equation*}
    \nu_0 = \delta_{x=b} \delta_{v=0}.
  \end{equation*}
  Then the distance is $\mathcal{W}_2(\mu_0,\nu_0) = |a-b|_\T$.

  At time $t$ the spatial distribution of $\mu_t$ and $\nu_t$, interpreted in $\R$, is a
  Gaussian with variance $\Sigma_{AA}$ which by the explicit formula
  \Cref{eq:covariance-aa} can be bounded as
  \begin{equation*}
    \Sigma_{AA}(t) \le C_A t^2
  \end{equation*}
  for a constant $C_A$ and $t \le 1$.

  Hence for $d > 0$ and $t \le 1$ the spatial spreading is controlled
  as
  \begin{align*}
    \mu_t((\T\setminus[a-d,a+d])\times\R)
    &\le \frac{2 \Sigma_{AA}(t)}{d\sqrt{2\pi}}
      \exp\left( \frac{-d^2}{2\Sigma_{AA}^2(t)} \right) \\
    &\le C_1 \frac{t^2}{d} \exp\left( - C_2 \frac{d^2}{t^4} \right)
  \end{align*}
  for positive constants $C_1$ and $C_2$, where we have used the standard tail bound for the Gaussian distribution (see e.g. \cite[Lemma 12.9]{brownian-motion}).

  For any $d>0$ small enough that $a\pm d$ and $b\pm d$ do not
  wrap around the torus, any coupling between $\mu_t$ and $\nu_t$ must
  transfer at least the mass
  \begin{equation*}
    1-\mu_t((\T\setminus [a-d,a+d])\times\R)-\nu_t((\T\setminus[b-d,b+d])\times\R)
  \end{equation*}
  between $[a-d,a+d]$ and $[b-d,b+d]$.

  Hence the Wasserstein distance is bounded by
  \begin{equation*}
    \mathcal{W}_2^2(\mu_t,\nu_t)
    \ge (|a-b|_\T-2d)^2
    \left( 1 - 2C_1 \frac{t^2}{d} \exp\left(-C_2
        \frac{d^2}{t^4}\right)
    \right).
  \end{equation*}
  Taking $d=|a-b|_\T t^{3/2}$ for $t$ sufficiently small, this shows that
  \begin{equation*}
    \mathcal{W}_2^2(\mu_t,\nu_t)
    \ge |a-b|^2_\T
    (1-2t^{3/2})^2
    \left( 1 - \frac{2C_1}{|a-b|_\T} \sqrt{t}
      \exp\left(- \frac{C_2 |a-b|_\T^2}{t} \right)
    \right).
  \end{equation*}
  However, for all small enough positive $t$, we have
  \begin{equation*}
    (1-2t^{3/2})^2 > e^{-\gamma t/2}
  \end{equation*}
  and
  \begin{equation*}
    \left( 1 - \frac{2C_1}{|a-b|_\T} \sqrt{t}
      \exp\left(- \frac{C_2 |a-b|^2_\T}{t} \right)
    \right)
     > e^{-\gamma t/2}
  \end{equation*}
  contradicting the assumed contraction. For the second estimate we use
  $\exp(-c/t) \le (1+c/t)^{-1} =
  t/(c+t)$.
\end{proof}

\section{Co-adapted couplings}\label{sec:markovian}
\subsection{Existence}\label{subsec:markovian:existence}
For \Cref{thm:markovian-coupling-existence} we construct a
reflection/synchronisation coupling using the drift-corrected
positions $Y^i_t$. As the positions are on the torus we can use a
reflection coupling until $Y^1_t$ and $Y^2_t$ agree. Afterwards, we
use a synchronisation coupling which keeps $Y^1_t = Y^2_t$ and reduces
the velocity distance.

For a formal definition let $((X^1_0,V^1_0),(X^2_0,V^2_0))$ be a
coupling between $\mu$ and $\nu$ obtaining the MKW distance (the
existence of such a coupling is a standard result, see
e.g. \cite[Theorem 4.1.]{villani2009optimal}). For a Brownian motion
$W^1_t$ let $(X^1_t,V^1_t)$ be the strong solution to
\eqref{eq:kinetic-fokker-planck-SDE} and define $(X^2_t,V^2_t)$ as the
strong solution with the reflected driving Brownian motion
\begin{equation*}
  W^2_t=\begin{cases}
    -W_t^1& t\le T\\
    W_t^1-2W^1_T&t>T.
  \end{cases}
\end{equation*}
with the stopping time $T=\inf\{t\ge 0: Y^1_t=Y^2_t\}$ with $Y^i_t$
from \eqref{eq:drift-correct-coordinates}. For the analysis we
introduce the notation
\begin{align*}
  M_t&=Y^1_t-Y^2_t\\
  Z_t&=V^1_t-V^2_t.
\end{align*}
Then by the construction the evolution is given by
\begin{align}
  \label{eq:coupling-sde-distance}
  \dd M_t &= \frac{2}{\lambda} 1_{t\le T} \dd W^1_t, \\
  \label{eq:coupling-sde-vel}
  \dd Z_t &= -\lambda Z_t\dd t+2\cdot 1_{t\le T}\dd W^1_t,
\end{align}
where $M_t$ evolves on the torus $\T$.

As a first step we introduce a bound for $T$.
\begin{lemma}
  The stopping time $T$ satisfies
  \begin{equation}\label{eq:T-distribution}
    \mathbb{P}(T>t|M_0)=\frac4\pi\sum_{k=0}^\infty\frac1{2k+1}\exp\left(-\frac{(2k+1)^2}{2\lambda^2L^2}t\right)\sin\left(\frac{(2k+1)|M_0|_{\T}}{2L}\right ).
  \end{equation}
\end{lemma}
\begin{proof}
  As $M_t$ evolves on the torus, $T$ is the first exit time of a
  Brownian motion starting at $M_0$ from the interval
  $(0,2\pi L)$. See \cite[(7.14-7.15)]{brownian-motion}, from which
  the claim follow after rescaling to incorporate the $2/\lambda$
  factor.
\end{proof}
\begin{remark}
  The second expression in \eqref{eq:T-distribution} is obtained by
  solving the heat equation on $[0,2\pi L]$ with Dirichlet boundary
  conditions and initial condition $\delta_{M_0}$.
\end{remark}
\begin{lemma}
  \label{thm:T-upper-bound}
  There exists a constant $C$ such that for any $t>0$ the following holds
  \begin{equation}\label{eq:T-upper-bound}
    \mathbb{P}(T>t|M_0)\le C|M_0|_\T(1+ t^{-1/2})e^{-t/(2\lambda^2L^2)}.
  \end{equation}
\end{lemma}
\begin{proof}
Using \eqref{eq:T-distribution} and the inequality $\sin(x)\le x$ for $x\ge0$, we have
\begin{equation*}
\begin{aligned}
\mathbb{P}(T>t|M_0)&\le \frac4\pi e^{-t/(2\lambda^2 L^2)}\sum_{k=0}^\infty \frac{|M_0|_{\T}}{2L}\frac{2k+1}{2k+1}e^{-4k^2t/(2\lambda^2L^2)}\\
&\le \frac{2}{\pi L}|M_0|_\T e^{-t/(2\lambda^2
  L^2)}\left(1+\int^\infty_0e^{-4u^2t/(2\lambda^2L^2)}\dd u\right)\\
&=\frac{2}{\pi L}|M_0|_\T e^{-t/(2\lambda^2 L^2)}\left(1+\sqrt{\frac{\pi}{8t/(\lambda^2L^2)}}\right)\\
&\le C|M_0|_\T(1+ t^{-1/2})e^{-t/(2\lambda^2 L^2)}
\end{aligned}
\end{equation*}
where on the second line we have bounded the sum by an integral.
\end{proof}
Using these simple estimates, we now study the convergence rate of the
coupling.
\begin{lemma}\label{thm:coupling-time-bound}
  There exists a constants $C$ such that for any $t\ge0$ we have the
  bound
\begin{equation*}
\mathbb{E}\left[|M_t|^2_\T+|Z_t|^2\middle|(Z_0,M_0)\right]\le |Z_0|^2e^{-2\lambda t}+\begin{cases}
C|M_0|_\T e^{-2\lambda t}& 2\lambda<1/(2\lambda^2L^2)\\
C|M_0|_\T(1+t)e^{-2\lambda t}& 2\lambda=1/(2\lambda^2L^2)\\
C|M_0|_\T e^{-t/(2\lambda^2L^2)}&2\lambda>1/(2\lambda^2L^2).
\end{cases}
\end{equation*}
\end{lemma}
\begin{proof}
  Without loss of generality we may assume that $Z_0$ and $M_0$ are
  deterministic in order to avoid writing the conditional expectation.

  Applying Itō's lemma, we find from \eqref{eq:coupling-sde-vel} that
  \begin{equation*}
    \dd|Z_t|^2=-2\lambda |Z_t|^2\dd t+4\cdot 1_{t\le T}Z_t\dd W^1_t+2\cdot 1_{t\le T}\dd t.
  \end{equation*}
  After taking expectations we see that
  \begin{equation}\label{eq:sde-for-|Z_t|^2}
    \frac{\dd}{\dd t}\mathbb{E}|Z_t|^2=-2\lambda \mathbb{E}|Z_t|^2+2\mathbb{P}(t\le T).
  \end{equation}
  By explicitly solving \eqref{eq:sde-for-|Z_t|^2} and using
  \Cref{thm:T-upper-bound}, we obtain
  \begin{equation*}
    \begin{aligned}
      \mathbb{E}|Z_t|^2&=|Z_0|^2e^{-2\lambda t}+2e^{-2\lambda t}\int^t_0e^{2\lambda s}\mathbb{P}(s\le T)\,\dd s\\
      &\le |Z_0|^2e^{-2\lambda t}+ C|M_0|_\T e^{-2\lambda t}\underbrace{\int^t_0e^{(2\lambda-1/(2\lambda^2L^2))s}(1+s^{-1/2})\,\dd s}_{=:I_t}.
    \end{aligned}
  \end{equation*}
  Let us bound $I_t$. As the integrand is locally integrable, we have
  for a constant $C$
  \begin{equation*}
    \begin{aligned}
      I_t&\le  C\left(1+\int^t_0e^{(2\lambda-1/(2\lambda^2L^2))s}\,\dd s\right).
    \end{aligned}
  \end{equation*}
  Here the $s^{-1/2}$ term can be bounded by 1 for $s>1$ and for
  $s \le 1$ the additional contribution can be absorbed into the
  constant.  To bound the remaining integral we consider three cases:
\begin{itemize}
\item $2\lambda<1/(2\lambda^2L^2)$: The integral (and $I_t$) are uniformly bounded, $I_t\le C$.
\item $2\lambda=1/(2\lambda^2L^2)$: The integrand is equal to $1$ and $I_t\le C(1+t)$.
\item $2\lambda>1/(2\lambda^2L^2)$: The integrand grows and $I_t\le C(1+e^{(2\lambda-1/(2\lambda^2L^2))t})$.
\end{itemize}
In each case we multiply $I_t$ by $e^{-2\lambda t}$ to obtain the decay rate. In the first two cases this gives the dominant term with $|M_0|_\T$ (as opposed to $|Z_0|$) dependence, while in the last case it is lower order than the $e^{-t/(2\lambda^2 L^2)}$ decay we obtain from $\mathbb{E}|M_t|^2_\T$ below.

Next let us consider $\mathbb{E}|M_t|^2_\T$. Using the finite diameter
of the torus we have the simple estimate
\begin{equation*}
\mathbb{E}|M_t|^2_\T\le \pi^2L^2\mathbb{P}(T>t).
\end{equation*}
For $t\ge1$ (say), we can use \Cref{thm:coupling-time-bound}, to obtain
\begin{equation*}
  \mathbb{E}|M_t|^2_\T\le C|M_0|_\T e^{-t/(2\lambda^2L^2)}
  \quad \text{for $t\ge 1$}.
\end{equation*}
This leaves when $t\le 1$ where \eqref{eq:T-upper-bound} blows up. We instead use the martingale property of $M_t$. Without loss of generality we may assume that $M_0\in[0,\pi L]$. Then as $M_t$ is stopped at $T$ we know that $M_t\in [0,2\pi L]$ for all $t\ge0$. Hence, for any $t\ge0$,
\begin{equation*}
\mathbb{E}|M_t|^2_\T\le \mathbb{E}|M_t|^2\le 2\pi L\mathbb{E}M_t=2\pi L M_0=2\pi L|M_0|_\T
\end{equation*}
by the martingale property. Combining the $t\le1$ and $t\ge1$ estimates we have
\begin{equation*}
  \mathbb{E}|M_t|_\T^2\le C|M_0|_\T e^{-t/(2\lambda^2L^2)}
  \quad
  \text{for $t\ge 0$}.
\end{equation*}
This together with the bound for $\mathbb{E}|Z_t|^2$ provides the claimed bounds of the lemma and completes its proof.
\end{proof}

By the equivalence of the norms from $(X,V)$ and $(Y,V)$, this is the
required coupling for \Cref{thm:markovian-coupling-existence}.

\subsection{Optimality}\label{subsec:markovian:optimality}

In order to show \Cref{thm:markovian-coupling-optimality}, we focus on
the drift-corrected positions $Y^1_t$ and $Y^2_t$ which behave like
time-rescaled Brownian motion on the torus. For their quadratic
distance we prove the following decay bound.

\begin{proposition}\label{prop:lower-bound-for-markovian-couplings}
  Suppose there exist functions $\alpha: (0,\pi L] \mapsto \R^+$ and
  $\beta : [0,\infty) \mapsto \R^+$ with $\beta \in L^1([0,\infty))$,
  such that, for any $z \in (0,\pi L]$ there exist two standard
  Brownian motions $W^1_t$ and $W^2_t$ on the torus $\T=\mathbb{R}/(2\pi L\mathbb{Z})$
  with respect to a common filtration such that $|W^1_0 - W^2_0| = z$,
  and for $t \in \R^+$ it holds that
  \begin{equation*}
    \E[|W^1_t-W^2_t|^2_\T]
    \le (\alpha(z))^2 \beta(t).
  \end{equation*}
  Then with a constant $c$ only depending on $L$, the function
  $\alpha$ satisfies the bound
  \begin{equation*}
    \alpha(z)\ge c\norm{\beta}_{L^1([0,\infty))}^{-1/2}\sqrt{z}.
  \end{equation*}
\end{proposition}

From this \Cref{thm:markovian-coupling-optimality} follows easily.
\begin{proof}[Proof of \Cref{thm:markovian-coupling-optimality}]
  Fix $z \in (0,\pi L]$ and consider the initial distributions
  $\mu = \delta_{X=0} \delta_{V=0}$ and
  $\nu = \delta_{X=z} \delta_{V=0}$. Between $\mu$ and $\nu$, there is
  only one coupling and $\mathcal{W}_2(\mu,\nu) = z$.

  If there exists a co-adapted coupling
  $((X^1_t,V^1_t),(X^2_t,V^2_t))$ satisfying the bound, then
  $Y^1_{t/\lambda^2}$ and $Y^2_{t/\lambda^2}$ are Brownian motions on
  the torus with a common filtration. Moreover,
  \begin{equation*}
    \E[|Y^1_t-Y^2_t|^2_\T]
    \le C\, \E[|X^1_t-X^2_t|^2_\T+|V^1_t-V^2_t|^2]
  \end{equation*}
  for a constant $C$ only depending on $\lambda$. Hence we can apply
  \Cref{prop:lower-bound-for-markovian-couplings} to find the claimed
  lower bound for $\alpha$.
\end{proof}

For the proof of \Cref{prop:lower-bound-for-markovian-couplings}, we
first prove the following lemma.
\begin{lemma}\label{lem:trivial-lower-bound-for-couplings}
  Given two Brownian motions $W^1_t$ and $W^1_t$ on the torus with a
  common filtration, then there exists a numerical constant $c$ such
  that
  \begin{equation*}
    \E[|W^1_t-W^2_t|_\T^2]
    \ge c\, e^{-2t/L^2} \E[|W^1_0-W^2_0|_\T^2].
  \end{equation*}
\end{lemma}
\begin{proof}
  The natural (squared) metric $|x-y|^2_{\mathbb{T}}$ on the torus is
  not a global smooth function of $x,y\in\mathbb{R}$ as it takes $x,y$
  $\operatorname{mod}2\pi L$. Therefore we introduce the equivalent
  metric
  \begin{equation*}
    d^2_\T(x,y)=L^2\sin^2\left(\frac{x-y}{2L}\right),
  \end{equation*}
  which is a smooth function of $x,y\in\mathbb{R}$. Moreover, the
  constants of equivalence are independent of $L$, i.e. there exist
  numerical constants $c_1$ and $c_2$ such that
  \begin{equation*}
    c_1 |x-y|^2_\T \le d^2_\T(x,y)
    \le c_2 |x-y|^2_\T.
  \end{equation*}

  Now consider $H_t$ defined by
  \begin{equation*}
    H_t = L \sin\left(\frac{W^1_t-W^2_t}{2L}\right)
    \exp\left(\frac{[W^1-W^2]_t}{4L^2}\right).
  \end{equation*}
  As $W^1_t$ and $W^2_t$ are Brownian motions, their quadratic
  variation is controlled as $[W^1-W^2]_t \le 4t$.  By Itō's lemma
  \begin{equation*}
    \dd H_t = \frac{1}{2}
    \cos\left(\frac{W^1_t-W^2_t}{2L}\right)
    \exp\left(\frac{[W^1-W^2]_t}{4L^2}\right)
    \dd (W^1-W^2)_t.
  \end{equation*}
  Hence Itō's isometry shows that
  \begin{align*}
    \E |H_t|^2
    &= \E |H_0|^2 + \E \int_0^t \frac{1}{4}
    \cos^2\left(\frac{W^1_t-W^2_t}{2L}\right)
    \exp\left(\frac{[W^1-W^2]_t}{2L^2}\right)
      \dd [W^1-W^2]_t \\
    &\le \E |H_0|^2 + \E \int_0^t
    \exp\left(\frac{2t}{L^2}\right)
      \dd t \\
    &< \infty.
  \end{align*}
  Therefore, $H_t$ is a true martingale and by Jensen's inequality
  \begin{equation*}
    \E[ |H_t|^2 ] \ge \E[ |H_0|^2 ].
  \end{equation*}
  Using the equivalence of two metrics, we thus find the required
  bound
  \begin{align*}
    \E[|W^1_t-W^2_t|_\T^2]
    &\ge c_2^{-1}
      \E\left[|H_t|^2
      \exp\left(-\frac{[W^1-W^2]_t}{2L^2}\right)
      \right] \\
    &\ge c_2^{-1} \E\left[|H_0|^2\right]
      \exp\left(-\frac{2t}{L^2}\right) \\
    &\ge c_1 c_2^{-1} \E[|W^1_0-W^2_0|_\T^2]
      \exp\left(-\frac{2t}{L^2}\right). \qedhere
  \end{align*}
\end{proof}

With this we approach the final proof.
\begin{proof}[Proof of \Cref{prop:lower-bound-for-markovian-couplings}]
  Fix $a \in (0,1)$, let $z \in (0,\pi L]$ be given, and by
  symmetry assume without loss of generality that
  $W^1_0 - W^2_0 = |W^1_0 - W^2_0| = z$. Then define the stopping time
  \begin{equation*}
    T = \inf\{t \ge 0: W^1_t - W^2_t \not \in (az,\pi L)\}.
  \end{equation*}
  The distance can be directly bounded as
  \begin{equation*}
    \E[|W^1_t-W^2_t|^2_\T] \ge \mathbb{P}[T \ge t] (az)^2.
  \end{equation*}
  As $\beta$ is integrable, it must decay along a subsequence of times
  and thus $T$ must be almost surely finite.

  As $W^1_t$ and $W^2_t$, considered on $\R$, are continuous
  martingales, their difference is also a continuous martingale. By
  the construction of the stopping time, the stopped martingale
  $(W^1-W^2)_{t \wedge T}$ is bounded by $\pi L$ and the optional
  stopping theorem implies
  \begin{equation*}
    \mathbb{P}[W^1_T-W^2_T=\pi L]
    = \frac{z-az}{\pi L - az}.
  \end{equation*}

  Since Brownian motions satisfy the strong Markov property, we find
  together with \Cref{lem:trivial-lower-bound-for-couplings}
  \begin{align*}
    \E \int_0^\infty |W^1_t-W^2_t|^2_\T \dd t
    &\ge \E \int_T^\infty |W^1_t-W^2_t|^2_\T \dd t  \\
    &\ge \mathbb{P}[W^1_T-W^2_T=\pi L] c\, (\pi L)^2 \int_0^\infty
      e^{-2t/L^2} \dd t \\
    &\ge \frac{z-az}{\pi L - az}
      c\, (\pi L)^2 \frac{L^2}{2} \\
    &\ge C_a z
  \end{align*}
  for a constant $C_a$ only depending on $a$ and $L$.

  On the other hand, integrating the assumed bound gives
  \begin{equation*}
    \E \int_0^\infty |W^1_t-W^2_t|^2_\T \dd t
    \le (\alpha(z))^2 \int_0^\infty \beta(t) \dd t
    \le (\alpha(z))^2 \norm{\beta}_{L^1([0,\infty))}.
  \end{equation*}
  Hence
  \begin{equation*}
    C_a z \le (\alpha(z))^2 \norm{\beta}_{L^1([0,\infty))}
  \end{equation*}
  which is the claimed result.
\end{proof}

\bibliographystyle{hplain}
\bibliography{lit}
\end{document}